\documentclass[11pt]{amsart}
\usepackage{tikz}
\usepackage{pgfplots}
\usepackage{amsthm}
\usepackage{amsxtra}
\usepackage{amssymb}
\usepackage{graphicx}
\usepackage{comment}
\usepackage{url}
\usepackage{color}
\usepackage{amsfonts}
\usepackage{textcomp}
\usepackage{perpage,enumerate}
\usepackage[colorlinks=true, citecolor=red]{hyperref}
\usepackage{caption}
\usepackage{subcaption}
\DeclareMathAlphabet{\mathpzc}{OT1}{pzc}{m}{it}
\usepackage{lmodern}

\usepackage{sidecap}

\DeclareMathOperator*{\supp}{supp}
\DeclareMathOperator*{\WF}{WF}

\DeclareMathOperator*{\graph}{graph}

\DeclareMathOperator*{\Id}{Id}

\usepackage{chngcntr}

\graphicspath{{Figures/}}

\newtheorem{theorem}{Theorem}

\numberwithin{prop}{section}

\numberwithin{corol}{section}

\newtheorem{lemma}{Lemma}
\numberwithin{lemma}{section}

\numberwithin{conjecture}{section}

{\theoremstyle{definition}

\numberwithin{defin}{section}
}
\numberwithin{figure}{section}

\newcommand{\RR}{\mathbb R}

\newcommand{\bl}{\begin{flushleft}}
\newcommand{\el}{\end{flushleft}}
\newcommand{\br}{\begin{flushright}}
\newcommand{\ert}{\end{flushright}}
\newcommand{\bc}{\begin{center}}
\newcommand{\ec}{\end{center}}

\newcommand{\numList}{\begin{enumerate}}
\newcommand{\enumList}{\end{enumerate}}

\newcommand{\e}{\epsilon}

\newcommand{\mc}[1]{\mathcal{#1}}

\theoremstyle{remark}
\newtheorem{remark}{Remark}
\newtheorem{definition}{Definition}

\renewcommand{\S}{\mc{S}\ell}

\newcommand{\Lap}{\Delta}
\newcommand{\RP}{\mathbb{RP}}
\newcommand{\ZZ}{\mathbb{Z}}
\newcommand{\NN}{\mathbb{N}}
\newcommand{\SC}{\operatorname{SC}}
\newcommand{\minP}{\blue{P}}
\newcommand{\nonminP}{\blue{\widetilde{P}}}
\newcommand{\py}{\pi_{_{\!Y\!\times\! Y}}}
\newcommand{\ps}{\pi_{_{\!S^2\!\times \!S^2}}}
\newcommand{\DD}{\mathsf{D}}

\newcommand{\ep}{\epsilon}

\newcommand{\abs}[1]{{\left\lvert{#1}\right\rvert}}

\newcommand{\smallabs}[1]{{\lvert{#1}\rvert}}

\newcommand{\red}[1]{{#1}} 
\newcommand{\blue}[1]{{#1}}

\usepackage{upgreek}

\DeclareMathOperator{\singsupp}{sing supp}

\title{On non-diffractive cones}

\author{Jeffrey Galkowski}
\address{Department of Mathematics, University College London, London, UK}
\email{j.galkowski@ucl.ac.uk}
\author{Jared Wunsch}
\address{Department of Mathematics, Northwestern University, Evanston, IL, USA}
\email{jwunsch@math.northwestern.edu}
\begin{document}
\maketitle

\begin{abstract}
A subject of recent interest in inverse problems is whether a corner must diffract fixed frequency waves. 
\red{We study the related question of which} cones $[0,\infty)\times Y$ which do not diffract high frequency waves. We prove that if $Y$ is analytic and does not diffract waves at high frequency then every geodesic on $Y$ is closed with period $2\pi$. Moreover, we show that if $\dim Y=2$, then $Y$ is isometric to either the sphere of radius 1 or its $\mathbb{Z}^2$ quotient, $\RP^2$. 
\end{abstract}

\section{Introduction}

A subject of recent interest in the study of inverse problems has been
the question of whether corners must diffract fixed-frequency
solutions of the Helmholtz equation with potential in $\RR^2$; here a
corner is the location of a singularity of the potential, which is of
the form of a smooth function times the indicator function of a
sector.  Affirmative answers to this question have been obtained under
various conditions by Bl{\aa}sten--P\"aiv\"arinta--Sylvester
\cite{BlPaSy:14} and P\"aiv\"arinta--Salo--Vesalainen \cite{PaSaVe:17}
(who treat certain kinds of conic singularities in $\RR^3$ as well).
More recently, results on diffraction by partially transparent
polygons and polyhedra have been obtained by Elschner--Hu
\cite{ElHu:18}.

In this note, we introduce a related problem that seems fundamental to
the theory of diffraction.  On a cone, perhaps the simplest setting in
which diffraction is know to occur, must there be nontrivial
diffraction at high frequency?  In posing the problem as a
high-frequency one, we restate it as a question about singularities of
solutions to the wave equation.  If we study the half-wave propagator
$\red{e^{-it\sqrt{\Lap}}},$ \blue{we ask: must there be singularities to the
solution other than those along (the closure of) the geodesics missing
the cone tip, i.e., those predicted by geometric optics in its
na\"ivest form?}  \blue{Passing to \red{the} frequency-domain via Fourier
transform, the existence of these singularities implies nontrivial
asymptotics, as the frequency parameter tends to infinity, in regions
not predicted by geometric optics away from the cone tip.  Hence the 
question under consideration here is equivalent to one of high-frequency asymptotics in
stationary scattering.} 

Our main theorem, admittedly a very partial result in
the desired direction, is that if a real-analytic cone exhibits no
diffraction in this sense, then its link must have the property that
every geodesic is $2\pi$-periodic.  In the special case when the link
has dimension $2$ (and is still analytic)  we are further able to show
that the link must be $S^2$ equipped with its standard round
metric of circumference $2\pi,$ or else $\RP^2,$ its $\ZZ_2$-quotient.
Some remarks on conjectured stronger results may be found below.

We now state our results more precisely.

\begin{definition}
  A cone $C(Y)$ over a Riemannian manifold $(Y,h)$ of dimension $d-1$ is the $d$-manifold
  $$
\red{C(Y)}=[0,\infty)_x \times Y
$$
whose interior is equipped with the metric
$$
g = dx^2 +x^2 h.
$$
\end{definition}
\blue{\noindent Thus from the point of view of metric geometry, in the cone
  $C(Y),$ all points $(0,y),\ y \in Y$ are identified.}

\blue{\begin{remark}For brevity, the results below will all be stated for cones $C(Y),$
which are sometimes referred to as \emph{product cones}.  We remark,
however, that in view of \cite[Theorem~3.2]{FoWu:17}, our results apply
equally to diffraction by more general conic metrics.  These are
nondegenerate metrics on the interior of a manifold with boundary
which near the boundary take the form
$$
g=dx^2 +\red{x^2}h(x,y,dx,dy),
$$
where $x$ is a boundary defining function and $h$ is a smooth
symmetric $2$-cotensor that restricts to be a
metric on the boundary.  Boundary components thus become cone points,
and the results of \cite{FoWu:17} show that the leading order
contribution to the diffracted wave can be determined from the case of
a model product
cone obtained by freezing coefficients at the boundary after making an
appropriate choice of boundary defining function.\end{remark}}

\begin{definition}
We say that $C(Y)$ is \emph{non\red{-}diffractive} if
$$
\singsupp \kappa(e^{-it\sqrt{\Lap}}) =\overline{ \big\{ p,p'\colon p,p' \text{ are endpoints of a geodesic of length } \smallabs{t} \text{ in } \red{C(Y)}^\circ\big\}},
$$
where $\kappa(A)$ denotes the Schwartz kernel of the operator $A$ \red{and $B^\circ$ denotes the interior of $B$}.
Otherwise, we say $C(Y)$ is \emph{diffractive}.  (Here $\Lap$ denotes
the Friedrichs extension of the nonnegative Laplace-Beltrami operator from $\mathcal{C}_c^\infty(\red{C(Y)^\circ})$.)
\end{definition}
It is known that in general there are additional ``diffracted'' singularities of this Schwartz kernel, at $$\DD_t \equiv \{p,p'\colon x(p)+x(p')=\smallabs{t}\};$$ indeed there is a conormal singularity along this set, degenerating near its intersection with the set of endpoints of geodesics in $\red{C(Y)}^\circ,$ which always carries singularities.  We remark that this intersection occurs exactly at the set
$$
\big\{p,p' \colon x(p)+x(p')=\smallabs{t},\ y(p), y(p') \text{ endpoints of a geodesic of length } \pi \text{ in } Y\big\}.
$$

It follows from the work of Cheeger--Taylor \cite{ChTa:82}, \cite{ChTa:82a} (see \cite[Corollary 2.3]{FoWu:17}) that the \emph{principal symbol} of the diffracted wave on $\DD_t$ is a nonvanishing multiple of
the Schwartz kernel of the operator
$$
\exp\left(-i\pi \sqrt{\Lap_Y+\frac{(d-2)^2}{4}}\,\right),
$$
where $\Lap_Y$ is the (positive definite) Laplacian on the link $Y$ of the cone, with
respect to the metric $h.$
Setting $$\nu=\sqrt{\Lap_Y+\frac{(d-2)^2}{4}},$$ we thus find that
\emph{a sufficient condition for $C(Y)$ to be diffractive} is that for
some $y\in Y,$ $\kappa(e^{-i\pi \nu}\delta_y)$ should have support outside the distance sphere of radius $\pi$ centered at $y.$  It is this condition that we exploit in proving the following.

\begin{theorem}\label{theorem:zoll}
Let $C(Y)$ be non\red{-}diffractive, and $Y$ real analytic. Then every geodesic on $Y$ must be periodic with (not necessarily minimal) period $2 \pi.$
  \end{theorem}
  \begin{remark}
    \blue{Many manifolds exist on which all geodesics are periodic with the
    same period:} in addition to the compact \blue{rank one} symmetric spaces and their quotients, there is a
    menagerie of so-called Zoll manifolds which enjoy this
    property---see \cite{Be:78} for detailed discussion. 
\end{remark}

    Conversely, we remark that if $Y$ is a spherical space form, i.e.,
    the quotient of $S^{d-1}$ with the standard metric on the \emph{unit} sphere by the fixed-point-free action of a \red{finite}
    subgroup of $G\subset O(d),$ then $C(Y)$ is the quotient of
    $\RR^{d}$ by the action of $G,$ blown-up at the origin (i.e.,
    viewed in polar coordinates).  The method of images then shows
    that $C(Y)$ is non\red{-}diffractive, since the Schwartz kernel of
    $e^{-it\nu}$ on $C(Y)$ may be obtained by averaging over
    the action of $G$ the corresponding Schwartz kernel on $\RR^d,$
    where ordinary propagation of singularities along geodesics holds
    true. In the case of $d=2,$ i.e., $\dim Y=1$,  $e^{-i\pi\nu}$ can be calculated explicitly as in the work of Hillairet~\cite{Hil02} and it is easy to verify that these are the only non-diffractive links. In
    fact \emph{we conjecture that these are the only examples,} even
    in the smooth category: if $C(Y)$ is non\red{-}diffractive and $Y$ merely
    $\mathcal{C}^\infty,$ then we conjecture that $Y$ must be a
    spherical space form.  This conjecture seems out of reach for the
    moment.

  Returning to the analytic case, we have been able to verify our conjecture in the
  case of dimension $2,$ ruling out Zoll manifolds \red{that are not spherical space forms}.
  \begin{theorem}\label{theorem:2d}
    Let $C(Y)$ be non\red{-}diffractive with $Y$ analytic and $\dim Y=2.$
    Then $Y$ is either $S^2$ or $\RP^2$ equipped with its standard metric.
  \end{theorem}
  We emphasize that by ``standard metric'' on $S^2$ or $\RP^2$ we do \emph{not} mean ``standard
  metric up to scale,'' but rather the metric on the
  \emph{unit} sphere in $\RR^3$ and its $\ZZ_2$-quotient respectively;
  spheres and projective spaces of other sizes \emph{do} diffract (as
  our proof shows). \red{Unlike in many other geometric situations, the scaling of the metric plays a role since it corresponds to the size of the ``opening'' of the cone.}

  In order to clarify these distinctions, we will use the notation
  $S^2_a$ and $\RP^2_a$ for the sphere equipped with the round metric
  of \emph{circumference} $a$ and its $\ZZ_2$-quotient, respectively.
  Hence $S^2_{2\pi}$ and $\RP^2_{2\pi}$ are the standard sphere and
  projective space.
  \blue{We introduce the non-standard terminology that a $\nonminP_a$
  manifold is one on which all geodesics are periodic with common
  period $a,$ while we follow \cite{Be:78} in letting $\minP_a$
  manifolds denote those $\nonminP_a$ manifolds on which $a$ is the
  \emph{minimal} common period.}
Thus, $S^2_a$ is a $\minP_a$ surface while
  $\RP^2_a$ is a $\minP_{a/2}$ surface.

\section*{Acknowledgments}
The authors are grateful to Steve Zelditch for helpful
discussions. Thanks also to the anonymous referees for many helpful
comments. J.G.\ is grateful to the National Science Foundation for
support under the Mathematical Sciences Postdoctoral Research
Fellowship DMS-1502661 and DMS-1900434.  J.W.\ was partially supported by NSF grant
DMS--1600023.  J.W.\ thanks the MBLWHOI Library for the use of its
reading room.

\section{Proof of Theorem~\ref{theorem:zoll}}
Let $\Phi_t$ denote geodesic flow for time $t$ on $S^*Y,$ i.e., the time-$t$ flow generated by the Hamilton vector field of $(1/2)\smallabs{\xi}^2_g,$ restricted to the unit cotangent bundle.  Let $\pi_{_{\!Y}}$ denote the projection $S^*Y \to Y.$

Let $$K\equiv \kappa(e^{-i \pi\nu}).$$  
Recall that a necessary condition for $Y$ to be non-diffractive is
$$
\supp K \subset \big \{ y,y'\colon y,y' \text{ are endpoints of a geodesic of length } \pi \big\}.
$$
(Standard propagation of singularities results \cite{DuHo:72} show that
the \emph{singular support} of $K$ lies in the latter set.)

Since $Y$ is analytic, we note that in order to show that all
geodesics are periodic with period $2\pi$, it suffices to show that on
a nonempty open set \blue{in $S^*Y,$} 
$$
\Phi_{2\pi}=\Id.
$$
Hence our strategy is to show that the support condition for $e^{-i\pi\nu}$ implies the existence of these closed geodesics.

Consider first the manifold
$$
\Lambda \equiv \graph(\Phi_\pi) \subset S^*Y \times S^*Y
$$
A key observation is now that $\WF K =\Lambda'$ (see, e.g., \cite[Theorem 1]{DuGu:75}) where 
$$
\Lambda':=\big\{(x,\xi,y,\eta)\mid (x,-\xi,y,\eta)\in \Lambda\big\}.
$$
\blue{Setting
$$
\begin{aligned}
  \Psi,\Psi': S^*Y &\to \Lambda\\
  \Psi(y,\eta) &= (y,\eta,\Phi_\pi(y,\eta))\\
    \Psi'(y,\eta) &= (y,-\eta,\Phi_\pi(y,\eta))\\
  \end{aligned}
  $$
  we thus express $\Lambda'$ as the range of $\Psi'.$

\blue{We now consider the projection,
$$
\py \Lambda \subset Y\times Y,
$$
which is where $\supp K$ lives, by hypothesis.  We remark that
$\py \Lambda$ is certainly not guaranteed to be a smooth manifold.
However, since our hypotheses imply that $\Lambda$ is analytic,
certainly $\py\Lambda$ is subanalytic, by definition.  A theorem of
Gabrielov \cite{Ga:68}, later rediscovered by Hironaka \cite{Hi:73}
and Hardt \cite{Ha:75} then implies that $\py \Lambda$ is a
\emph{stratified} space, and in particular, contains as an open
subset, $F$, a maximal-dimensional embedded submanifold (note that the
``semianalytic shadows'' in \cite{Ha:75} are synonymous with
subanalytic sets). We will
employ a slight strengthening of this statement, also following from the results
of \cite{Ha:75}.}
 

\red{ \begin{lemma}
 \label{l:stratified}
 There is an open subset $F\subset \py \Lambda$ such that 
 \begin{enumerate}
 \item $F$ is a maximal-dimensional embedded submanifold.
 \item The set $\tilde{F}= \Psi^{-1}\py^{-1}(F)\subset S^*Y$ is open
 \item  For $\rho \in\tilde{F}$, $\operatorname{rank}d\py d\Psi(\rho)=\dim F.$
 \end{enumerate}
 \end{lemma}
 \begin{proof}
 Recall that a \emph{stratification} of a manifold $M$ is a locally finite collection $\mc{S}$ of connected, embedded open submanifolds such that $\sqcup_{S\in\mc{S}}S=M$ and
 \begin{gather*}
 \text{ if }S,T\in \mc{S}\text{ and }T\cap \partial S\neq \emptyset,\text{ then }\dim T<\dim S,\, T\subset \partial S.
 \end{gather*}
 Let $\Theta:=\py\circ \Psi$. By~\cite[Corollary 4.4]{Ha:75}, since
 $\Theta:S^*Y\to Y\times Y$ is an analytic mapping of real analytic
 manifolds, there is a stratification, $\mc{S}$, of $S^*Y$ and
 $\mc{T}$ of $Y\times Y$ such that for $S\in \mc{S}$, \blue{$\Theta(S) \in
 \mc{T},$ with
 $$
 \operatorname{rank}d\Theta|_{S}= \dim\Theta(S).
 $$}
 Define  $k:=\sup_{S^*Y} \operatorname{rank}d\Theta$ and let $\tilde{F}\in \mc{S}$ such that   
 $$
\operatorname{rank}d\Theta|_{\tilde{F}}=k,\qquad \dim \tilde{F} =2n-1
 $$
 Then $F:=\Theta(\tilde{F})\in \mc{T}$ is an embedded submanifold of dimension $k$. Moreover, $\py\Lambda$ is contained in a finite union of submanifolds of dimension $\leq k$ and hence $F$ has maximal dimension. 

Now, since $\mc{T}$ is a stratification, and $\Theta(S)\in\mc{T}$ for $S\in \mc{S}$,
$$
\bigcup_{\substack{S\in \mc{S}\\S\neq S_k}}\Theta(S)\cap F=\emptyset.
$$
In particular, $\Theta^{-1}(F)=\tilde{F}$ and the proof is complete.
 \end{proof}
}

\red{Let $F$, $\tilde{F}$ as in Lemma~\ref{l:stratified}. Then, by construction, for $\rho\in\tilde{F}$, $d\py d\Psi:T_{\rho}S^*Y \to T_{\py \circ\Psi(\rho)} F$ is surjective. }
\begin{lemma}\label{lemma:conormal}
Suppose that $d\py d\Psi:T_{\rho}S^*Y \to T_{\py \circ\Psi(\rho)} F$ is surjective. Then $\Psi'(\rho)\in SN^*\!F$. 
\end{lemma}
\begin{proof}
Fix any $\rho_0 \equiv (y,\eta) \in S^*Y$ satisfying the hypotheses.  We need to show that all vectors in
$T_{\red{\py\circ}\Psi(\rho_0)} F$ are annihilated by pairing with $\Psi'(\rho_0).$

By hypothesis,
$$
\begin{aligned}
T_{\red{\py\circ}\Psi(\rho_0)} F &= \big\{ d\py \circ d\Psi( V)\colon V \in
               T_{\rho_0}S^*Y\big\}\\
  &= \big\{(d\pi(V), d\pi\circ  d\Phi_\pi (V))\colon V \in
               T_{\rho_0}S^*Y\big\},
\end{aligned}
    $$
    hence we need to show that the pairing of a vector of this form with
    $\Psi'(\rho_0)$ vanishes, i.e., (letting square bracket denote the
    pairing of a covector with a vector) that
\begin{equation}\label{pairing1}
-\rho_0[d\pi (V)]+\Phi_\pi(\rho_0)[d\pi\circ  d\Phi_\pi( V)]=0,\quad \text{ for all } V
\in T_{\rho_0} S^*Y.
\end{equation}

    Now we investigate the quantity $\Phi_\pi(\rho_0) [d\pi\circ d\Phi_\pi (V)]$.
    For any $V\in T_{\rho_0}S^*Y,$ choose $\rho:(-\e,\e)\to S^*Y$ with
    $\rho(0)=\rho_0$ and $\partial_s\rho|_{s=0}=V$. Next, define
    $\Gamma(s,t)=\pi(\Phi_t(\rho(s)))$. Then
    $J(t):=\partial_s\Gamma(s,t)|_{s=0}=d\pi \circ d\Phi_t V$ is a
    Jacobi field along $\gamma(t):=\Gamma(0,t)$ and
$$
\Phi_t(\rho(0))[J(t)]=\langle \dot{\gamma}(t) ,J(t)\rangle_g.
$$

Since $J$ is a Jacobi field, $\partial_t^2\langle \dot{\gamma}(t)
,J(t)\rangle_g=0$ (see \cite[p.288]{Le:18}).
Moreover, we compute (using symmetry of the connection---\cite[Lemma 6.2]{Le:18})
\begin{align*}
\partial_t\langle \dot{\gamma}(t),J(t)\rangle_g|_{t=0}&= \langle \dot{\gamma}(0),D_t J(0)\rangle_g\\
&=\langle \dot{\gamma}(0),D_s d\pi H_p(\rho(s))|_{s=0}\rangle_g\\
&=\langle d\pi H_p(\rho(0)),D_s d\pi H_p(\rho(s))|_{s=0}\rangle_g\\
&=\frac{1}{2}\partial_s \langle d\pi H_p(\rho(s)),d\pi H_p(\rho(s))\rangle_g|_{s=0}
\end{align*}
Now, in coordinates, we have $d\pi H_p= g^{ij}\xi_i\partial_{x_j}$ and therefore, since $\rho(s)\in S^*Y$,
$$
\langle d\pi H_p(\rho(s)),d\pi H_p(\rho(s))\rangle_g=g^{ij}\xi_i(s)\xi_j(s)\equiv 1.
$$
Therefore, $\partial_t\langle \dot{\gamma}(t) ,J(t)\rangle_g|_{t=0}=0$. We have now shown that for any $V\in T_{\rho_0}S^*Y$, $\Phi_t(\rho_0)[d\pi \circ d\Phi_t V]$ is constant and in particular, 
$$
\rho_0[d\pi \circ d\Phi_t V]-\Phi_\pi(\rho_0)[d\pi \circ d\Phi_\pi V]=0,
$$
thereby establishing \eqref{pairing1}.
\end{proof}   }

Our hypotheses are that $\supp K \subset \py \Lambda,$ hence in a neighborhood $V$ of any point in $F$, $\supp K \subset F.$  Since $F$ is a smooth \red{embedded sub}manifold, we thus know that on $V$, we may express
$$
K=\sum \delta^\alpha(u)\phi_\alpha(y)
$$
where $u=(u_1,\dots,u_k)$ are defining functions for $F$, $y$ complete $u$ to a local coordinate system, \red{and $\phi_\alpha\in \mc{D}'(F)$}. \red{Moreover, by Lemma~\ref{lemma:conormal}, for $\rho\in \tilde{F}=\Psi^{-1}\py^{-1}(F)$, $\Psi'(\rho)\in SN^*F$. In particular, 
$$
\WF(K)\cap\py^{-1}(F)=\Lambda'\cap \py^{-1}(F)\subset SN^*F
$$ 
which implies $\phi_\alpha\in C^\infty(F)$.}

 Such a distribution has the property that {its} wavefront set is invariant under the negation map on fibers:
$$
(y,\eta,y',\eta') \in \WF K\cap \py^{-1} (F\cap V)\Longrightarrow (y,-\eta,y',-\eta') \in \WF K\cap \py^{-1} (F\cap V).
$$
Thus, since $\WF K=\Lambda',$
$$
(y,\eta,y',\eta') \in \Lambda \cap \py^{-1} (F\cap V) \Longrightarrow (y,-\eta,y',-\eta') \in \Lambda \cap \py^{-1} (F\cap V).
$$
This precisely means that for $(y,\eta) \in \pi_L(\Lambda \cap\py^{-1}(F\cap V)),$
$$
\Phi_{\pi}(y,\eta) = -\Phi_{\pi}(y,-\eta)=\Phi_{-\pi}(y,\eta)
$$
(with negation interpreted as acting on the fibers).  Hence
$$
\Phi_{2\pi}(y,\eta)=(y,\eta).
$$
\red{Now set
  $U=\pi_L(\Lambda \cap\py^{-1}(F\cap V)).$
  Since $\pi_L:\Lambda \to S^*Y$ is
  bijective, and $\Lambda\cap \py^{-1}(F\cap V)$ is open} we have proved the desired periodicity of geodesics on a
  nonempty open set in $S^*Y.$ \qed

\section{Proof of Theorem~\ref{theorem:2d}}

By Theorem~\ref{theorem:zoll}, $Y$ is a $\nonminP_{2\pi}$ surface.  Thus, it
is diffeomorphic to either $S^2$ or $\RP^2$---see \cite[Section
4.3]{Be:78}.

We begin with the case where $Y$ is diffeomorphic to
$S^2.$  As in the proof of Theorem~\ref{theorem:zoll}, we consider
$\ps\Lambda\subset S^2\times S^2,$ the projection of the graph of
time-$\pi$ geodesic flowout in $S^*( S^2);$ we again use crucially that
this is a stratified space.  Since the dimension of $\Lambda$ itself
is $3$ and since \red{projections onto the left and right  factor of $Y$} of $\py\Lambda$ are surjective,
the dimension of the maximal stratum of $\ps\Lambda$ may only be $2$
or $3.$  If it is $3,$ then there is an open set, $F$ in $\ps\Lambda$ that
is a submanifold of $S^2\times S^2$ of codimension-$1,$ so that the
Schwartz kernel of the propagator $K$ is locally given by
\begin{equation}\label{propagator}
K=\sum_{\red{|\alpha|\leq M}} \delta^{(\alpha)}(u)\phi_\alpha(y)
\end{equation}
where now $u \in \RR$ is locally a defining function for $\ps\Lambda$, \red{$M<\infty$ and $\phi_\alpha \in C^\infty(F)$}.

We will need a slightly stronger consequence of~\cite[Theorem 1]{DuGu:75} than that $(\WF(K))'=\Lambda$. In particular, we need that
\begin{equation}\label{Kregularity}
\begin{aligned}
  ({\WF}^{-1}(K))'&=\Lambda,\\
    {\WF}^{-1-\ep}(K)&=\emptyset \text{ for all }  \ep>0,\\
  \end{aligned}
\end{equation}
where $\WF^s$ denotes the $s$-wavefront set, i.e., $\rho\notin
\WF^s(u)$ if and only \red{if} there exists $A\in \Psi^0$ so that $Au\in H^s$
and $\sigma(A)(\rho)\neq 0$. Now, let $x_0\in F$ and $V$ a
neighborhood of $x_0$ so that~\eqref{propagator} is valid on $V$. Let
$\chi\in C_c^\infty(V)$ with $\chi(x_0)=1$. Then,
by~\eqref{propagator} \begin{equation}\label{chiK}
  \chi K\in \bigcup_{\e>0}H^{-1/2-j-\e}\setminus
H^{-1/2-j},\end{equation} where $j$ is the largest $\smallabs{\alpha}$
such that the coefficient $\phi_\alpha$ in \eqref{propagator} is
nonvanishing on $\supp \chi.$  On the other hand,
$$
\begin{aligned}
  T^*_{x_0}(S^2\times S^2)\cap {\WF}^{-1}(K)&\neq\emptyset,\\
    T^*_{x_0}(S^2\times S^2)\cap {\WF}^{-1-\ep}(K)&=\emptyset.
  \end{aligned}
$$
Either the first or the second of these statements contradicts
\eqref{chiK} depending on whether $j=0$ or $j \geq 1.$

We conclude from this contradiction that in fact the dimension of the
maximal stratum is $2.$
\blue{On the other hand, the rank of the projection from
  $\Lambda\subset S^*S^2 \times S^*S^2$ to $S^2$ in the first
  factor already has rank $2$. Hence in order for the stratum
  dimension not to exceed $2,$ it must be the case that $d_\xi
  \pi_{S^2}\Phi_\pi (x,\xi)=0$ for all $x, \xi \in S^*(S^2).$}
This means that
$\pi_{_{\!S^2}} \Phi_\pi(S^*_x (S^2))$ is a \emph{single point} for
each $x \in S^2;$ for brevity we denote this \blue{point}
$\Phi_\pi(x).$ \blue{Since $Y$ is a $\nonminP_{2\pi}$ manifold, by
 positivity of the injectivity radius of a compact manifold, there is
 some positive minimal common period, hence $Y$ is a
  $\minP_{2\pi/k}$ manifold for some positive integer $k$. We
 now consider separately the cases where $k$ odd and even.}

\blue{\noindent\textsc{Case 1: $k$ odd.}}  It has been shown by
Gromoll--Grove \cite{GrGr:81} that on $Y$ diffeomorphic to $S^2$, the
$\minP_{a}$ condition implies that $Y$ is an $\SC_{a}$ manifold (again
in the terminology of \cite[Section 7.8]{Be:78}), which is to say, all
geodesics have minimal period exactly $a$ and are without
self-intersection (``simple'').  If $Y$ is a $\minP_{2\pi/k}$ manifold
for $k$ \emph{odd}, we thus conclude from \cite{GrGr:81} that
$\Phi_\pi(x)\neq x$ for all $x,$ as otherwise this would contradict
simplicity of the geodesics.  \red{\begin{lemma}\label{l:blaschke}
    Suppose that $Y$ is as above and $Y$ is a $\minP_{2\pi/k}$ surface
    for some $k$ odd. Then, $Y$ is a Blaschke surface.
\end{lemma}
\begin{proof}
We recall from
  \cite[Theorem 5.43]{Be:78}
 that among several equivalent definitions of a Blaschke surface is that \emph{the cut locus is spherical}, which is to say the
  distance to the first cut point is independent of direction at each
  point.  For $Y$ a $\minP_{2\pi/k}$ surface with $k$ odd, $\Phi_{\pi/{{k}}}(x)$ has distance $\pi/{{k}}$ from $x$
  for all $x,$ since otherwise a geodesic from $x$ would pass through
  $\Phi_{\pi/{{k}}}(x)$ at time $t_0\in (0, \pi/{{k}})$ and then would
  self-intersect at time $\pi/{{k}},$ contradicting the simplicity of the
  geodesics from \cite{GrGr:81}.  But then every geodesic must in fact
  be minimizing up to time-$\pi/{{k}},$ as a failure to be minimizing would
 allow us to construct a continuous, piecewise smooth curve from $x$ to
 $\Phi_{\pi/{{k}}}(x)$ of length shorter than $\pi/{{k}}.$  Hence the cut-radius is
 exactly $\pi/{{k}},$ at every point in every direction, and our surface is
 indeed Blaschke.
 \end{proof}}
By Lemma~\ref{l:blaschke} together with the
resolution of the Blaschke Conjecture \cite{Gr:63} (cf.\ \cite[Theorem
5.59]{Be:78}), we conclude that $Y=S^2_{2\pi/{{k}}}.$

\blue{We now rule out nonstandard spheres, with $k \neq 1.$
We will identify spheres of all radii with one another using standard
polar coordinates, and note that for all $k,$ on $S^2_{2\pi/k},$
$\Phi_{\pi/k}(x)=-x,$ the antipode.
}
\blue{Now observe that since $e^{-i\pi \nu}\delta_x$ is supported at $\Phi_{\pi}(x)=-x$ and $K$ takes the form~\eqref{propagator},
$$
e^{-i\pi \nu}\delta_x=c_1\delta_{-x}
$$
for some $c_1\neq 0$. Moreover, $c$ is independent of $x$ since isometries act transitively on $S^2_{2\pi/k}$. Now, since $S^2_{2\pi}$ is non-diffractive,
$$
\sum_{\ell=0}^\infty \sum_{m=-\ell}^\ell e^{-i\pi
  (\ell\red{-}\frac{1}{2})}\overline{Y_\ell^m (y)}Y_\ell^m (x)=c_2\delta_{-x}(y)
$$
for some $c_2\neq 0$. Thus, there is $c\neq 0$ such that for all $x,y$
$$
\sum_{\ell=0}^\infty \sum_{m=-\ell}^\ell e^{-i\pi
  (\ell\red{-}\frac{1}{2})}\overline{Y_\ell^m (y)}Y_\ell^m (x)=c
\sum_{\ell=0}^\infty \sum_{m=-\ell}^\ell e^{-i\pi \sqrt{ k^2\ell(\ell-1)+\frac{1}{4}}}\overline{Y_\ell^m (y)}Y_\ell^m (x).
$$
Since the $Y_\ell^m$ form an orthonormal basis for $L^2$, this implies
that for all $\ell \in \NN,$
$$
\sqrt{k^2\ell(\ell-1)+\frac{1}{4}}-\ell+\frac{1}{2} \in \beta+ 2 \ZZ 
$$
for some fixed real number $\beta\in [0, 2).$
We note, though, that
$$
\sqrt{k^2\ell(\ell-1)+\frac{1}{4}}-\ell+\frac{1}{2} =k\ell -\frac 12 k
-\frac{k^2-1}{8k \ell} -\ell +\frac 12 + O(\ell^{-\red{2}}),
$$
and, \red{by Lemma~\ref{l:noInteger} (below),} this quantity cannot have constant fractional parts as $\ell
\to \infty$ unless $k=1.$

Thus, $Y=S^2_{2\pi}$.}  \blue{This finishes the case where $k$
is odd.}

\red{
\begin{lemma}
\label{l:noInteger}
Let $p,q\in\mathbb{Z}$, $q> 0$, $b,c\in \RR$ and $c\neq 0$ and define 
$$
\alpha_\ell:= \frac{p}{q}\ell+b+c\ell^{-1}+O(\ell^{-2}),\qquad \ell=1,2,\dots
$$
Suppose that $\{j_k\}_{k=1}^\infty\subset \mathbb{Z}$ with $j_k\underset{k\to \infty}{\longrightarrow} \infty$. Then there are $\ell, m$ such that 
$$
\alpha_{j_\ell}-\alpha_{j_m}\notin \mathbb{Z}. 
$$
\end{lemma}
\begin{proof}
Let 
$$
f(\ell)  =\frac{p}{q}\ell+b+\frac{c} {\ell}.
$$
Then, there is $L>0$ such for $j\geq J$, $\alpha_j$ satisfies
$$
\abs{\alpha_j-f(j)} \leq \frac{|c|}{3j }.
$$
Hence, for $\ell, m\geq L$, letting $e_\ell=\alpha_\ell-f(\ell),$ $e_m=\alpha_m-f(m),$ we obtain
\begin{equation}\label{mess}
\begin{aligned}
\alpha_m-\alpha_\ell &= \frac{p}{q}m+b+\frac{c}{ m}+e_m-(\frac{p}{q}\ell+b+\frac{c}{
                      \ell})-e_\ell\\
&=\frac{p}{q}(m-\ell) +\frac{c}{m}-\frac{c}{ \ell}+e_m-e_\ell.
\end{aligned}
\end{equation}
Now, our estimates on $e_m,$ $e_\ell$ easily give
$$
 |e_m-e_\ell|\leq \frac{|c|}{3m}+\frac{|c|}{3\ell}.
$$
Now, since $j_k\to\infty$, there is $M>0$ such that for $m\geq M$, $j_m>\max(L,4|c|{q})$. Fix $m\geq M$. Let $\ell\geq M$ such that $j_\ell\geq 4j_m$. Then, we have
$$
\alpha_{j_m}-\alpha_{j_\ell}\,\operatorname{mod}\frac{1}{q}\blue{\ZZ}=\frac{c}{j_m}-\frac{c}{ j_\ell}+e_{j_m}-e_{j_\ell}\,\operatorname{mod}\frac{1}{q}\blue{\ZZ},
$$
and
$$
0<\frac{|c|}{\blue{3j_m}}\leq \Big|\frac{c}{j_m}-\frac{c}{ j_\ell}+e_{j_m}-e_{j_\ell}\Big|\leq \frac{3|c|}{\blue{2j_m}}\leq \frac{3}{4q}.
$$
Hence the fractional part of $\alpha_{j_m}-\alpha_{j_\ell}$ is nonzero. 
\end{proof}}


\blue{\noindent\textsc{Case 2: $k$ even.}}
We now assume that $Y$ is a $\minP_{2\pi/k}$ surface for
some \emph{even} integer $k,$ hence that $Y$ is a $\nonminP_\pi$ surface; we will
then derive a contradiction.

For each $y \in S^2,$ $e^{-i \pi \nu} \delta_y$ is by hypothesis
supported at the flowout of $S^*_y (S^2),$ which we now know to be the
point $y$ itself.  Thus $e^{-i \pi \nu} \delta_y$ must equal $\psi(y) \delta_y$ for some function
$\psi;$ more generally this tells us that \begin{equation}\label{propagatorequation}
    e^{-i\pi \nu} f=\psi f\end{equation} for
every $f \in L^2.$  Applying \eqref{propagatorequation} to $f=\phi_j,$ an
eigenfunction of $\Lap_Y$ with eigenvalue $\lambda^2,$ tells us that
\begin{equation}
  \label{eq:1}
  e^{-i\pi \sqrt{\lambda_j^2+1/4}}=\psi(y).
\end{equation}
The left side is of course constant, so $\psi$ is in fact constant,
and all of these values must agree, i.e., there exists $\beta\in \red{[0,2)}$
such that for all $\lambda_j^2$ in the spectrum of $\Lap_Y,$
\begin{equation}\label{integralconstraint}
\sqrt{\lambda_j^2+1/4}\equiv \beta \bmod 2 \ZZ;
\end{equation}
equivalently this is just the statement that the spectrum of $\nu$ lies in
$\beta + 2\ZZ.$

Now in order to derive a contradiction, we turn to the strong results known about spectral asymptotics of
Zoll surfaces; this argument is based on the fact
that \blue{the spectrum of a $\minP_{\pi/k}$ manifold} must closely resemble
that of $S^2_{\pi/k},$ which is indeed diffractive.  Duistermaat--Guillemin \cite{DuGu:75},
Weinstein \cite{We:75}, and Colin de Verdi\`ere \cite{Co:79}
have obtained very precise estimates of the \emph{clustering} of the
eigenvalues of such a Zoll surface.  Thus, e.g.,
\cite[Corollaire~1.2]{Co:79} (see also \cite{Gu:77}, \cite{We:77})
    shows that there is $M>0$ so that the spectrum of $\Lap_Y\red{+\frac{1}{4}}$ is entirely contained in
a union of intervals
$$
I_n=\big[ 4(n+\alpha/4)^2-M, 4(n+\alpha/4)^2+M\big],
$$
where $\alpha$ is the Maslov index of all the $\pi$-periodic geodesics;
the (crucial)\footnote{\red{These factors will later give rise to our contradiction and arise from the rescaling of a $\nonminP_{2\pi}$ metric to a $\nonminP_{\pi}$ metric.}} factors of $4$ arise since we are dealing with a $\nonminP_\pi$ surface
rather than a $\nonminP_{2\pi}$ surface as in \cite{Co:79}.  \red{Since the eigenvalues of $\Delta_Y\red{+\frac{1}{4}}$, $\lambda_\bullet^2\red{+\frac{1}{4}}$, lie in $I_n$,} the square
roots $\red{\sqrt{\lambda_\bullet^2+\frac{1}{4}}}$ of the eigenvalues lie in intervals
$$
J_n =\big[ 2(n+\alpha/4)-C/n, 2(n+\alpha/4)+C/n\big].
$$
On the
other hand, for $n$ large, the constraint \eqref{integralconstraint}
implies that \emph{each interval $I_n$ can contain at most one
  eigenvalue} \blue{(possibly with high multiplicity)}. \red{Indeed, for $n$ large enough, $2C/n<2$ and hence the interval $J_n$ has length less than 2 and contains at most one element of the form~\eqref{integralconstraint}.}
We
have thus reduced to the situation studied by Zelditch in \cite{Ze:96}
of a \emph{maximally degenerate Laplacian}; Zelditch proves
\cite[Theorem C]{Ze:96} that this places a yet stronger constraint on
the locations of the eigenvalues and that there is an operator $A$
with spectrum in $\NN$ such that
$$
\Lap_Y= 4\big(A+ \frac 12 \big)^2- 1 +S
$$
with $S$ a smoothing operator; here again we have rescaled by a factor
of $4$ since we are dealing with a $\nonminP_\pi$ surface.  Hence the eigenvalues
$\sqrt{\lambda_\bullet^2+1/4}$ of $\nu$ are all of the form
\blue{\begin{equation}\label{eigenvalues}
\sqrt{4(\ell+1/2)^2-3/4}+O(\ell^{-\infty})=2\ell+1-\frac{3}{16 \ell} +O(\ell^{-2}).
\end{equation}}
Recall on the other hand that they are in $\beta+2\ZZ$ by
\eqref{integralconstraint}.
\blue{\red{By Lemma~\ref{l:noInteger},} these two constraints are incompatible, i.e.\ solutions to
\eqref{eigenvalues} cannot asymptotically differ by even integers. 
}

Hence we have ruled out all $\nonminP_\pi$ manifolds, and
completed the case of $Y$ diffeomorphic to $S^2.$

To finish the proof, we
now turn to the (easier) case when $Y$ is diffeomorphic to $\RP^2.$
In this case, \blue{Lin--Schmidt~\cite[Theorem 2]{LinSchmidt} shows that since all geodesics on $Y$ are closed, $Y=\RP^2_a$ for some $a>0$. (Cf.\ Green's proof of Blaschke's
Conjecture~\cite{Gr:63}.) Next, since $Y$ is a $\nonminP_{2\pi}$ manifold,}
we know that $Y=\RP^2_{4\pi/k}$ for some $k \in\NN.$  We now rule out
all but $\RP^2_{2\pi}.$  To start, we know by the same argument as in
the sphere case that $\Phi_\pi(x)$ is a single point for each $x.$
This does not happen unless $k$ is even and at least $2.$  For $k$
even, $\RP^2_{4\pi/k}$ is a 
$\nonminP_\pi$-manifold so the same argument \red{as} in the sphere case tells us
that the spectrum of $\nu=\sqrt{\Lap_Y+1/4}$ lies in $\beta + 2\ZZ$ for $\beta$ fixed.
Next, the
spectrum of  $\RP^2_{4 \pi /k}$ is the set
$$
\red{\frac {k^2}{4}}2\ell(2\ell+1),\quad \ell \in \NN;
$$
thus the spectrum of $\nu$ is
\red{$$\bigg(\frac{k^2}{4} 2\ell(2\ell+1)+\frac{1}{4}\bigg)^{\frac{1}{2}}=k\ell+\frac{k}{\blue{4}}-\frac{k^2-4}{32k\ell}+O(\ell^{-2}).$$}
\blue{\red{By Lemma~\ref{l:noInteger},} these cannot have constant fractional part unless $k=2.$}\qed

      \bibliographystyle{abbrv} 
\bibliography{references.bib}

\end{document}